\documentclass[11pt]{article}
\usepackage{fullpage}
\usepackage{amsmath,amssymb, amsthm}
\usepackage{ulem}
\usepackage{cite}
\usepackage{multicol}
\usepackage{relsize}
\usepackage{authblk}

\usepackage{url}
\DeclareMathOperator{\sol}{div} 
\DeclareMathOperator{\curl}{curl}

\renewcommand{\div}{{\rm div}} 
		
\newcommand{\norm}[1]{\left\|#1\right\|}   
\newcommand{\abs}[1]{\left\lvert#1\right\rvert}  
\newcommand{\field}[1]{\mathbb{#1}}

\def\Mdiv{M_{\div\ufi}(m)}
\def\R{\field{R}} 
\def\T{\field{T}} 
\def\D{\field{D}}
\def\S{\field{S}} 
\newcommand{\rntou}[1]{\R^{#1}}    
\def\carky{\sp{\prime\prime}}
\def\Id{\field{I}} 
\def\epsil{\varepsilon}
\newcommand{\vektor}[1]{{\mathbf{#1}}}
\def\ro{\varrho}
\def\en{\vektor{n}}
\def\Phib{\boldsymbol\Phi}
\def\omeg{\boldsymbol{\omega}}
\def\ef{\vektor{F}}
\def\Gef{\mathcal{\widetilde G}}
\def\u{\vektor{u}}
\def\uf{\widetilde{\u}}
\def\ufi{\vektor{U}  }
\def\rof{\widetilde{\ro}}
\def\rf{\widetilde{r}}
\def\lam{{\ell}} 

\def\up{\vektor{V}}
\def\rp{R}

\newcommand{\de}[1]{\mathrm{d}{#1}}     
\def\dx{\de{x}}	
\newcommand{\inte}[1]{\int\limits_{\Omega}{#1}\, \dx  }
\newcommand{\inth}[1]{\int_{\partial\Omega}{#1}\, \de{S}  }
\newcommand{\refx}[1]{(\ref{#1})}

\def\qqquad{\qquad\quad}

\def\Pg{P_\mathsmaller{\mathsmaller{\nabla}}}
\def\PH{P_H}

\theoremstyle{plain}
\newtheorem{thm}{Theorem}
\newtheorem{lem}{Lemma}
\newtheorem{prop}{Proposition}

\newtheorem*{thm*}{Theorem}

\newcommand{\lnorm}[2]{\norm{#1}_{#2}}
\newcommand{\lnorms}[2]{{\bigl\|#1\bigr\|}_{#2}}

\title{Steady solutions to viscous shallow water equations.\\ 
The case of heavy water.}
\author[1]{\v{S}imon Axmann\thanks{axmas7am@karlin.mff.cuni.cz}}

\author[2]{Piotr Bogus\l{}aw Mucha\thanks{p.mucha@mimuw.edu.pl; Corresponding author}}

\author[1]{Milan Pokorn\'{y}\thanks{pokorny@karlin.mff.cuni.cz}\small}

\affil[1]{Charles University in Prague, Faculty of Mathematics and Physics}

\affil[2]{Institute of Applied Mathematics and Mechanics, University of Warsaw}

\begin{document}
\maketitle
{\bf Abstract.} In this note, we show the existence of regular solutions to the stationary version of the Navier--Stokes system for compressible fluids with a density dependent viscosity,
known as the shallow water equations. For arbitrary large forcing we are able to construct a solution, provided the total mass is sufficiently large.
The main mathematical part is located in the construction of solutions. Uniqueness is impossible to obtain, since the gradient of the velocity
is of magnitude of the force. The investigation is connected to the corresponding singular limit as Mach number goes to zero and methods for weak solutions to the compressible Navier--Stokes system.

\bigskip


{\bf key words:} steady compressible Navier-Stokes system, shallow water equation, low Mach number limit, density dependent viscosities, large data, existence via Schauder type fixed point theorem.

\section{Introduction and the main result}

The subject of the paper is  the following steady version of the Navier--Stokes system for compressible fluid
\begin{gather}
     \sol(\ro \u) = 0, \label{RKs} \\ 
 \sol(\ro \u\otimes\u) =\sol \T+\ro\ef,   \label{MEs}
\end{gather}
 where
 $\ro$ is density, $\u$ the velocity field, $\ef$ the specific external force, and the stress tensor $\T(\ro,\nabla\u) = \S(\ro,\nabla\u) - p(\ro)\Id$ with
  pressure $p(\ro) = \ro^{\gamma} $, and the viscous stress $\S$ satisfying the Stokes law for the Newtonian fluid 
$$\S(\ro,\nabla \u) = \mu(\ro)\bigl(\nabla\u +\nabla^{T}\u \bigr) + \lambda (\ro) \sol\u\Id,$$
  with the viscosity coefficients and pressure (some remarks to a more general situation will be given at the end of this section)
  $$
  \mu(\ro) = \ro, \qquad \lambda(\ro)=0, \qquad p(\ro) = \ro^{\gamma}.
  $$
  We denote the symmetric part of the velocity gradient by $\D(\u)=\frac{1}{2}\bigl(\nabla\u +\nabla^{T}\u \bigr)$.
  For the two dimensional equations and $\gamma=2$ the model coincides with the  well known shallow water equations. However,
  we concentrate on the three dimensional version of the system  with general $\gamma>1$.
The system is supplemented with the slip boundary condition for the velocity 
\begin{gather}
\u\cdot\en =0,\label{imper}\\
\en\cdot \S(\ro,\nabla\u)\cdot\boldsymbol{\tau}^k   +f  \u \cdot \boldsymbol{\tau}^k =0 \text{ at }\partial\Omega,\label{slip} 
\end{gather}
where  $\boldsymbol{\tau }^k,\:k=1,2$ are two linearly independent tangent vectors to $\partial\Omega$,  $\en$ denotes the normal vector and the constant $f$ is the non-negative friction coefficient.
Furthermore, we assume  the total mass is prescribed,
\begin{equation}\label{total}
 \int_\Omega \ro \ \dx = |\Omega| m,
\end{equation}
with $m>0$, a given number. Domain $\Omega$ is assumed to be bounded, three-dimensional with a smooth, say $C^2$, boundary.

Our goal is to construct a regular solution of class $(\u,\ro) \in W^{2,p}(\Omega) \times W^{1,p}(\Omega)$ for arbitrary, possibly large, external force $\ef$, provided the total mass is large enough. The slip boundary conditions allow 
to use the Helmholtz decomposition and effectively use the information carried by the effective viscous flux. Here we meet the theory of weak solutions to the compressible Navier--Stokes system and approaches from \cite{NoPa94}. 
The a priori estimate is relatively easy to get, but the main difficulty lies in the construction of the solutions. Since the solutions are essentially large, we have to modify standard applications of the Schauder fixed point 
theory. Our main theorem can be compared to results of Choe and Jin \cite{ChJi00} (see also \cite{Doetal15} for the heat-conducting case). Authors study there the low Mach number problem for the steady system with Dirichlet boundary conditions.
Working in the $H^m$ framework, they obtain also large solutions as a perturbation of corresponding incompressible flows. The statement of the problems (our and from \cite{ChJi00}) are similar for $\gamma =2$, but in the case of $\gamma \in (1,2)$ we obtain
essentially different asymptotics of the system.

\smallskip

Our main result  reads as follows.

\begin{thm} \label{main}
Let $\gamma >1$.  Suppose that $\Omega$ is a smooth bounded  domain in $\rntou{3},$ which is not axially symmetric, $\ef\in L^{p}(\Omega)$ for some $p\in(3,6)$ and $m$ 
is sufficiently large with respect to the norm of $\ef$. Then there exists at least one strong solution to the Navier--Stokes equations
\refx{RKs}--\refx{total} in the class $(\ro,\u)\in W^{1,p}(\Omega)\times W^{2,p}(\Omega).$
\end{thm}

First, let us  note  the result can be proved in the two dimensional case, too. 
The methods and estimates are the same (coefficients are slightly different, but considerations are easier).
Thus we leave this case.
We could also deal with the axially symmetric domain $\Omega$; however,
we would be required to put several artificial and technical assumptions.
 
The most important case $\gamma =2$ is just the model of shallow water. Due to the structure of the 
equations the system is related to the low Mach number limit \cite{Bre,Bre2}.
%
%
Concerning the known existence results of steady non-constant solutions near equilibrium, we refer to M. Padula\cite{Padu87},
H. Beir{\~a}o da Veiga \cite{Beir87} and R. Farwig \cite{Farw88, Farw89}, see also T. Piasecki \cite{PiasJDE}.
The corresponding low Mach number limit is extensively studied as well \cite{FeNo07,Alaz06,DeGr99}.

\smallskip 

Now, we  perform a formal analysis of our system. Provided the total mass is  large, we expect
that the density can be considered in the following form
\begin{equation*}
 \ro = m + r, \mbox{ \ \ with \ \ } \int_\Omega r \ \dx =0.
\end{equation*}
Our analysis will be based on this assumption. Provided $m$ dominates $r$, we meet the titled case of heavy density fluids. 
However, as we will see, this restriction does not  limit the magnitude of the gradient of the velocity. Hence, for large forces
we obtain turbulent flows, thus any uniqueness property is not possible to reach.
Using  \eqref{RKs}--\eqref{MEs} and the form of $\S$, we restate the system as follows
\begin{gather}
     m\sol \u  + \u\cdot\nabla r + r\sol\u  = 0, \label{RKs1} \\ 
(m+r)\u\cdot\nabla\u -m\Delta\u -    m \nabla\sol\u + \gamma (m+r)^{\gamma -1}\nabla r = 
  2  \D(\u) \nabla r+r\Delta\u +   r \nabla\sol\u+\ro\ef.   \label{MEs1}
\end{gather}
Formally,  norms of solutions are essentially smaller than $m$,  the problem transforms into the following one
\begin{gather}
     \sol \u = o(m), \label{RKs2} \\ 
\u\cdot\nabla\u -\Delta\u -    \nabla\sol\u + \gamma m^{\gamma -2}\nabla r = \ef+o(m),   \label{MEs2}
\end{gather}
where $o(m) \to 0$ as $m\to \infty$ in suitable norms. Here we see that the case $\gamma=2$, which corresponds to the low Mach number limit,
is distinguished among all case; the left-hand side is independent of $m$. In the case $\gamma \in (1,2)$ the norms of  constructed 
density  depend strongly on $m$. This picture illustrates the key difference to results from \cite{ChJi00}. The structure of system (\ref{RKs2})-(\ref{MEs2}) is more complex
and the linearization strongly depends on parameter $m$.

We skipped a possible generalization of the result in the case 
\begin{equation*}
 \mu(\ro) \sim \ro^{L} \mbox{ \ \ for \ \ } L >1, \mbox{\ and/or \ } \lambda(\ro) \sim \mu(\ro).  
\end{equation*}
Then, looking at our formal asymptotic \eqref{RKs2}--\eqref{MEs2}, we would obtain on the left-hand side of \eqref{MEs2} 
$$
m^{1-L}\u \cdot \nabla \u - \Delta \u -\nabla \sol \u
$$
which implies that the convective term is marginalized and we  arrive at the case of small solutions with 
obvious uniqueness, what is not our aim. Similarly, the proof for $\lambda(\ro)\sim \mu(\ro)$ follows the same lines as for $\lambda(\ro) \equiv 0$. The additional terms behave as the terms we deal with for $\mu(\ro)$.

\smallskip 

Within the paper we use the standard notation. By $\|\cdot\|_{m,p}$ we denote the norm of the Sobolev space $W^{m,p}(\Omega)$ defined over domain $\Omega$ for $m\in \mathbb{N}, p\in [1,\infty]$. Norms of functionals and of traces are displayed by 
their full symbols.

\section{A priori estimates}

In this section we construct the a priori estimate, which determines the class of regularity of our sought solutions.
Assume 
$$\ro=m+r,$$
where $\inte{r\:}=0$ and $\frac{1}{\abs{\Omega}}\inte{\ro\:} = m$, with $m$  large enough. For $p>3$ we define 
the following quantity
\begin{equation}
\Xi = m^{\gamma -2} \norm{r}_{{1,p}} + \norm{\u}_{{2,p}} \label{notation}
\end{equation}and consider solutions for which 
\begin{equation}
m\gg \Xi + \lnorm{\ef}{p}.\label{dense}
\end{equation}

System (\ref{RKs})--(\ref{MEs}) can be then rewritten as follows
\begin{gather}
     m\sol \u  + \u\cdot\nabla r + r\sol\u  = 0, \label{RKs3} \\ 
(m+r)\u\cdot\nabla\u -m\Delta\u  + \nabla p(m+r) = 
  2  \D(\u) \nabla r+r\Delta\u +   r \nabla\sol\u+\ro\ef.   \label{MEs3}
\end{gather}
The basic energy estimate reads
$$\inte{2m\abs{\D(\u)}^2} +\int_{\partial \Omega} f |\u|^2  {\mathrm{ d}} S = \inte{\ro\ef\cdot\u } 
-\inte{2r\abs{\D(\u)}^2},
$$
hence if we further assume  
$ m>2\lnorm{r}{\infty}$  (see \refx{dense}),  we obtain, due to Korn's inequality,
\begin{equation}
\lnorm{\u}{1,2} \leq C \lnorm{\ef}{6/5}. \label{EI}
\end{equation}
Note that if  $f= 0$, the assumption that the domain can not be 
axially symmetric is needed for Korn's inequality to hold \cite{MuJDE}.
%
%
 Further, we test the momentum equation with a function $-\Phib$, $\Phib=\mathcal{B}[r],$ where  $\mathcal{B}\sim\sol^{-1}$~denotes the Bogovskii operator, accordingly we have $\lnorm{\Phib}6\leq C\lnorm{\nabla\Phib}2\leq C\lnorm{r}2$. This yields 
\begin{equation*}
 (\gamma m^{\gamma-1}-\lnorm{r}\infty)\lnorm{r}2^2\\ \leq  C m\biggl(\lnorm{\nabla\u}2 \lnorm{\nabla\Phib}2 +  \lnorm{\u}3\lnorm{\nabla\u}2 \lnorm{\Phib}6  +  \lnorm{\ef}{6/5}\lnorm{\Phib}6 \biggr),  
\end{equation*}
 hence using Young's inequality and \refx{EI}
 $$ m^{\gamma-2}\lnorm{r}2^2  \leq \frac{C}{m^{\gamma-2}} \Bigl( {\lnorm{\ef}{{6}/{5}}^4}+ {\lnorm{\ef}{{6}/{5}}^2}   \Bigr).$$
 
 In order to recover the effective viscous flux, we  apply the Helmholtz decomposition for functions in $L^p(\Omega)$ with values in ${\rntou{3}}$. Linear operators read
\begin{equation}
\Pg: L^p(\Omega)\to W^{1,p}(\Omega) \qquad \text{and} \qquad \PH:L^p(\Omega)\to L_{\sol}^p(\Omega) \label{Helm}
\end{equation}
with the properties $\vektor{g} = \PH(\vektor{g}) + \nabla \Pg(\vektor{g})$, $\sol \vektor{g} = \Delta \Pg(\vektor{g}) $, and $\en \cdot\PH(\vektor{g})  = 0 $ on $\partial\Omega.$
We  estimate the solenoid and  gradient part of the momentum equation separately.

 First, applying the $\curl$-operator on \refx{MEs3} yields for $\omeg=\curl\u$
$$
 - m  \Delta \omeg =\:\curl\Bigl( - \ro\u\cdot\nabla\u 
  + \sol \bigl( 2r  \D(\u)\bigr)+\ro\ef  \Bigr) \quad \mbox{ in } \Omega 
$$
with the boundary conditions
$$
\left.\begin{array}{rl}
   \sol \omeg =& 0,\\ 
  \omeg\cdot\boldsymbol{\tau}_1 =& (f/(m+r)- 2\chi_2)\u\cdot\boldsymbol{\tau}_2,\\ 
   \omeg\cdot\boldsymbol{\tau}_2 =& (2\chi_1-f/(m+r))\u\cdot\boldsymbol{\tau}_2
	\end{array}
	\right\}
	\text{ on }\partial\Omega,
$$
 with $\chi_i$ denoting the curvatures corresponding to the directions $\boldsymbol{\tau}_i.$ 
 The form of boundary conditions in the above system comes from features of the slip boundary relations \cite{MuRa,MuJDE}.
 Thus, according to the elliptic regularity theory (see also \cite{MuPo07})
 $$m \norm{\omeg}_{{1,p}} \leq   C\Big(\norm{\curl\bigl( - \ro\u\cdot\nabla\u 
   + \sol \bigl( 2r  \D(\u)\bigr)+\ro\ef \bigr)}_{(W^{1,p{\sp{\prime}}}(\Omega))^*} +m\norm{ \u}_{W^{1-\frac{1}{p},p}(\partial\Omega)} \Big) ,$$
where $\norm{ \u}_{W^{1-\frac{1}{p},p}(\partial\Omega)}\leq C \norm{\u}_{W^{1,p}(\Omega)}$ and $(W^{1,p'}(\Omega))^*$ denotes the dual space to $W^{1,p'}_0(\Omega)$. Further, $\PH\u$ satisfies the overdetermined system
 \begin{align}
  \curl\PH\u = \: &\omeg\text{ in }\Omega,\nonumber\\
  \sol\PH\u =\: & 0\text{ in }\Omega,\label{rotdiv}\\
 \PH\u\cdot\en =\: & 0\text{ on }\partial\Omega,\nonumber
  \end{align}
 yielding \cite{Solo73, MuPo14} $ \lnorm{\nabla^2\PH\u}p\leq C  \norm{\omeg}_{1,p}. $
 Thus,
\begin{equation}
\lnorm{\nabla^2\PH\u}{p} \leq \frac{1}{m} C\bigl( \lnorm{\ro\u\cdot\nabla\u}p + \lnorm{\nabla \u}{\infty}\lnorm{\nabla r}p  +  \lnorm{r}\infty\lnorm{\nabla^2 \u}p 
 + \lnorm{\ro\ef}p+m\lnorm{\nabla\u}p\bigr).\label{Stokes}
\end{equation}

Similarly, the potential part of the momentum equation \refx{MEs3} 
reads\footnote{We denote $\{g\}_\Omega=\frac{1}{\abs{\Omega}}\inte{ g }  $. }
\begin{equation*}
p(\ro) - \{p(\ro)\}_\Omega - 2 m \sol\u = \Pg\bigl( \mathcal{G} + m \Delta \PH\u
 \bigr),\label{h10}
\end{equation*}
where we put 
\begin{equation}\label{h11}
 \mathcal{G}=- \ro\u\cdot\nabla\u 
  + 2  \sol\bigl(r\D(\u)\bigr)+\ro\ef.
\end{equation}
In our considerations we  keep in mind that $\Pg(\Delta\u)=\Delta\Pg\u + \Pg(\Delta\PH\u)$ and $\Delta \Pg\u = \sol\u.$

We  use the Taylor expansion, in order to observe
 \begin{equation}
  p(\ro) = (m+r)^\gamma = m^\gamma+ \gamma m^{\gamma-1}r + \frac{1}{2} p\carky(\xi) r^2,\label{TaylorL}
 \end{equation}
 where $\xi$ lies between $m$ and $m+r$, whence $\abs{p\carky(\xi) r^2}\leq C m^{\gamma-2}r^2$. Subtracting the average from \refx{TaylorL} yields
 \begin{equation*}
 p(\ro)-\{p(\ro)\}_\Omega = \gamma m^{\gamma-1}r + \frac{1}{2} \bigl( p\carky(\xi) r^2- \{p\carky(\xi) r^2 \}_\Omega \bigr) .
 \end{equation*}
Then  we  combine
 \begin{equation*}
  \gamma m^{\gamma-1}r-2 m \sol\u + \frac{1}{2} \bigl( p\carky(\xi) r^2- \{p\carky(\xi) r^2 \}_\Omega \bigr) = \Pg\Bigl( \mathcal{G} +  m \Delta \PH\u  \Bigr)\label{h31}
 \end{equation*}
 with the continuity equation
  $$ m\sol \u  + \u\cdot  \nabla r +  r\sol\u  = 0  ,$$
 in order to get
 \begin{equation}\label{h32}
  \gamma m^{\gamma-1}r  +2\nabla r \cdot \u   =  -2 r\sol\u+\Pg\Bigl( \mathcal{G} +  m \Delta \PH\u  \Bigr) - \frac{1}{2} \bigl( p\carky(\xi) r^2- \{p\carky(\xi) r^2 \}_\Omega \bigr).
 \end{equation}
Differentiating \eqref{h32}, we obtain
 \begin{multline}
  \gamma m^{\gamma-1}\nabla r  + 2\u\cdot\nabla\nabla r    =  -2 \nabla r\sol\u -2r\nabla\sol\u - 2\nabla \u \nabla r - \frac{1}{2} \nabla \bigl( p\carky(\xi) r^2\bigr)
 \\+\nabla\Pg\Bigl(  \mathcal{G}+ m \Delta \PH\u \Bigr).\label{gradrL}
 \end{multline}
 Note that $\Pg$ is continuous from $L^p$ to $W^{1,p}$, so $\nabla\Pg$ is actually a zero order operator. To obtain from \refx{gradrL} the required information about $\nabla r$, we test the $k$-th component of \refx{gradrL} by $\partial_kr\abs{\partial_kr}^{p-2}.$ The second term on the left hand side can be then rewritten using integration by parts as
 $$\inte{\u\cdot\nabla\partial_kr\abs{\partial_kr}^{p-2}\partial_kr}=-\frac1p\inte{\sol{\u}\abs{\partial_kr}^p};$$
  $ \abs{\nabla \bigl( p\carky(\xi) r^2\bigr)}\leq C m^{\gamma-2}\abs{r} \abs{\nabla r}.$
 Thus, we get due to the Poincar\'e inequality and the fact that $m\gg 1$
 \begin{equation}
  m^{\gamma-1}\lnorm{r}{1,p}\leq C\Bigl(\lnorm{\nabla r}p\lnorm{\nabla \u }{1,p}+ \lnorm{\mathcal{G}}p + m\lnorm{\nabla^2\PH\u}p  \Bigr).\label{h18L}
 \end{equation}
 The first term on the right-hand side can be put to the left-hand side for $\Xi\ll m^{\gamma-1}.$
 Moreover, using \refx{h31}, we  bound the potential part of the velocity. Since
  \begin{equation*}
  2m  \nabla\sol\u = \gamma m^{\gamma-1}\nabla r +  \frac{1}{2} \nabla \bigl( p\carky(\xi) r^2- \{p\carky(\xi) r^2 \}_\Omega \bigr) - \nabla \Pg\Bigl( \mathcal{G} + m \Delta \PH\u
    \Bigr),
  \end{equation*}
  we obtain for the quantity $\nabla\sol\u$ similar estimate, namely
  \begin{equation}
  m\lnorm{\nabla \sol\u}{p}\leq  C \bigl( m^{\gamma-1}
  \lnorm{\nabla r}p + \lnorm{\mathcal{G}}p + m\lnorm{\nabla^2\PH \u}p\bigr).\label{h20L}
  \end{equation}
 Putting together \refx{Stokes}, \refx{h18L} and \refx{h20L} yields 
 $$ \Xi \leq  \frac{C}{m}  \lnorm{\mathcal{G}}p  + C   \lnorm{\nabla\u}p . $$
 By (\ref{h11})
it is easy to see that the most restrictive term is the convective term. 
We  estimate it for $p\in(3,6]$ with interpolation and energy inequality \refx{EI} as follows
 \begin{equation*}
 \begin{split}
 \lnorm{\ro\u\cdot\nabla\u}p\leq \lnorm{\ro}\infty\lnorm\u6\lnorm{\nabla\u}{\frac{6p}{6-p}}\leq &\: \bigl( m+\lnorm r\infty \bigr)\lnorm\u6\lnorm{\nabla\u}2^
 {\frac{6-p}{3p}}\lnorm{\nabla\u}\infty^
 {\frac{4p-6}{3p}}  \\
 \leq &\:C m \lnorm{\ef}{6/5}^{\frac{2p+6}{3p}}\lnorm{\nabla^2\u}p^{\frac{4p-6}{3p}}.
 \end{split}
 \end{equation*} 
 To sum up, we get for $p<6$
 \begin{equation*}
 \Xi\leq C\Big( \lnorm{\ef}p+ {\lnorm{\ef}{6/5}^{\frac{2p+6}{6-p}}}
  \Big).
 \end{equation*}
Thus, under the assumption $\gamma>1,$ we obtain the a priori estimate
\begin{equation}
\lnorm{\nabla\u}{1,p}+m^{\gamma-2}\lnorm{r}{1,p}=\Xi\leq C_\ef.\label{CFL}
\end{equation}

The basic idea is to take $m$ sufficiently larger than the right-hand side of \refx{CFL}, id est 
\begin{equation}C_\ef\ll \min(m^{\gamma-1},m).\label{smallness2}\end{equation}
Finally, we  look back on the continuity equation \refx{RKs3}, and conclude from \refx{CFL} that 
$$\lnorm{\sol\u}p\leq 2\frac{C_\ef^2}{m^{\gamma-1}}.$$
It expresses how far we are from the incompressible flow.
Note that as $\gamma\to 1^+ $, condition \refx{smallness2}  requires larger and larger $m$. 
In particular, $\gamma=1$ would demand small external force.


\section{Approximation}
Let us denote the classes of regularity where  the solutions are searched for
 \begin{align*}
M_r(m)=&\biggl\{ f\in W^{1,p}(\Omega),\int_\Omega{f}\, \dx=0, \, m^{\gamma-2}(\lnorm{f}\infty+\lnorm{\nabla f}p)\leq C_\ef \biggr\},\\
M_\u(m)=&\Bigl\{ \vektor{f}\in W^{2,p}(\Omega,\rntou{3}), \:\vektor{f}\cdot\en={0} 
\text{ on }\partial \Omega, \Bigr.\nonumber\\ \Bigl.&\qqquad\qqquad
\lnorm{\nabla\vektor{f}}2 \leq E, \: \lnorm{\nabla \vektor{f}}\infty+\lnorm{ \vektor{f}}\infty+\lnorm{\nabla^2\vektor{f}}p 
\leq C_\ef,\:  m^{\gamma-1}\lnorm{\sol \vektor{f}}p \leq 2C_\ef^2 \Bigr\},
\end{align*}
where  $C_\ef$ is from \refx{h51}, and $E$ represents the upper bound for the kinetic energy, see \refx{EIf}.
However,  $M_\u(m)$ is not a compact subset of $W^{2,p}(\Omega)$.
Therefore, in order to perform in our last step a Schauder fixed point argument, we o introduce additionally another set, which is a closed subset of $W^{1,\infty}(\Omega)$, $M_\u(m) \subset \Mdiv$, namely
\begin{align*}
\Mdiv=&\Bigl\{ \vektor{f}\in  W^{1,\infty}(\Omega,\rntou{3}), \:\vektor{f}\cdot\en={0} \text{ on }\partial \Omega, \Bigr.\nonumber\\ \Bigl.&\qqquad\qqquad \lnorms{\nabla\vektor{f}}2 \leq E, \: \lnorms{\nabla \vektor{f}}\infty+\lnorm{ \vektor{f}}\infty   \leq C_\ef,\:  m^{\gamma-1}\lnorms{\sol \vektor{f}}p \leq 2C_\ef^2 \Bigr\}.
\end{align*} Our general strategy is as follows. We denote $\rof=m+\rf$.
 First, we fix $\ufi\in \Mdiv$, and $\rf \in M_r(m)$ and use the Leray--Schauder, as well as the Banach fixed point theorem to show the existence of a solution $(r,\u)\in M_r(m)\times M_\u(m) $ to the following system. 
\begin{gather}
 m\sol \u  + \sol( r \u)  = 0, \label{RKgiven} \\ 
\rof\ufi\cdot\nabla\u -\sol\bigl(2\rof\D(\u)\bigr)+ \gamma m^{\gamma-1}\nabla r  + \nabla R_m(\rf)=\rof\ef \text{ in }\Omega,\label{MEgiven}\\
\u\cdot\en=0,\qquad \en\cdot 2 \rof\D(\u)\cdot\boldsymbol{\tau}^k  + f  \u\cdot \boldsymbol{\tau}^k =0  \text{ on }\partial\Omega,\label{BCgiven}
\end{gather}
where, see \refx{TaylorL},
$$
R_m(\rf)=p(m+ \rf)-\gamma m^{\gamma -1}\rf -m^\gamma, \mbox{ \ \ and  \ \ } |R_m(\rf)|\leq Cm^{\gamma-2} \rf^2.
$$
The  uniqueness  for problem  \refx{RKgiven}--\refx{BCgiven} will be a consequence of the construction. Then,  fixing $\ufi\in \Mdiv$, we show via the Banach contraction principle that there exists a solution $(r,\u)\in M_r(m)\times M_\u(m) $ to the system
\begin{gather}
 m\sol \u  + \sol( r \u) = 0, \label{RKgiv} \\ 
(m+r)\,\ufi\cdot\nabla\u -\sol\bigl(2(m+r)\D(\u)\bigr)+ \gamma m^{\gamma-1} \nabla r  + \nabla R_m(r)=(m+r)\ef   \label{MEgiv}
\end{gather}
with boundary conditions \refx{imper}--\refx{slip}.
Finally, we  show the existence of a fixed point of the mapping $\mathcal{T(\ufi)}=\u$ in $M_\u(m)$ by means of the Schauder fixed point theorem.

We start with the following proposition  concerning problem (\ref{RKgiven})--(\ref{BCgiven}).
\begin{prop}\label{Prop1}
Suppose $\ufi\in \Mdiv $, $\uf\in M_\u(m)$, $\rf\in M_r(m)$ for $m$  sufficiently large, then there exists a solution $(r,\u)$ to problem \refx{RKgiven}--\refx{BCgiven} in the class $M_r(m)\times M_\u(m).$
\end{prop}

\begin{proof}
First, for a given $\vektor{G}\in L^{p}(\Omega)$ and $h\in W^{1-\frac{1}{p},p}(\partial\Omega)$ we study the problem
\begin{align}
 m\sol \u  +  \sol( r \uf)  = &\:  0, \label{rof} \\ 
 -m\Delta\u + \gamma m^{\gamma-1}\nabla r  =&\:   \vektor{G} - \rof\ufi\cdot\nabla{\u} \text{ in }\Omega,  \label{mof}\\
 \u\cdot\en= &\:0, \\ \en\cdot 2 m \D(\u)\cdot\boldsymbol{\tau}^k + f \u\cdot\boldsymbol{\tau}^k  = &\: h\text{ on }\partial\Omega,\quad\int_{\Omega}{r\:}\dx=0.  \label{bf}
\end{align}
\begin{lem}\label{l1}
Given  $\vektor{G}\in L^{p}(\Omega)$ and $h\in W^{1-\frac{1}{p},p}(\partial\Omega)$, there exists a unique solution to system \refx{rof}--\refx{bf} with $r\in W^{1,p}(\Omega)$, $\u\in W^{2,p}(\Omega)$.
\end{lem}
\begin{proof}[Proof of Lemma \ref{l1}] First note that the system is linear. We proceed in the following way. We fix   $\overline{r}\in W^{1,2}(\Omega)$ and use elliptic regularization of 
the continuity equation in order to get merely weak solution to the system with fully linearized continuity equation, then we use the Leray--Schauder argument to obtain a solution to \refx{rof}--\refx{bf}, and finally improve the regularity using the method of decomposition.

For $\epsil>0$ and $\overline{r}\in W^{1,2}(\Omega) $ we consider
\begin{align}
-\epsil \Delta r + \epsil r +   m\sol \u  +  \sol( \overline{r} \uf)  = &\:  0, \label{epsilr} \\ 
 -m\Delta\u + \gamma m^{\gamma-1}\nabla r +  \rof\ufi\cdot\nabla{\u} =&\:   \vektor{G}\text{ in }\Omega, \label{uwithr} \\
 \u\cdot\en= 0,\qquad \en\cdot\nabla r = &\: 0, \\
 \en\cdot 2 m \D(\u)\cdot\boldsymbol{\tau}^k + f \u\cdot\boldsymbol{\tau}^k  = &\: h \text{ on } \partial\Omega.  \label{epsilb}
\end{align}
It is a strictly elliptic  problem, hence the existence of a unique solution follows from the Lax--Milgram theorem; note that $\lnorm{\sol(\rof\ufi)}2\ll m $, so the convective term
is not problematic. Further,  using as test function for \eqref{epsilr} the function $\gamma m^{\gamma-2}r$ and for \eqref{uwithr} the function $\vektor{u}$, we get estimates
\begin{equation*}
\epsil m^{\gamma-2}\norm{r}_{1,2}^2  + m\norm{\u}_{1,2}^2 \leq C(\vektor{G},h,\ufi,\uf,\rf,\norm{\overline{r}}_{1,2})
\end{equation*} 
with $C$ independent of $\epsil$, and from \refx{epsilr} we conclude that actually $r\in W^{2,2}(\Omega)$. Therefore, we see that the mapping $\mathcal{T}:\overline{r}\mapsto r$ defined
through \refx{epsilr}--\refx{epsilb} is a continuous and compact mapping on $W^{1,2}(\Omega)$ for any $\epsil>0$. To apply the Leray--Schauder fixed point theorem, it remains to show that the possible fixed points 
\begin{equation*}
\lam\mathcal{T}(r)=r\label{fixpoint}
\end{equation*} 
are bounded in $W^{1,2}(\Omega)$ independently of $\lam\in[0,1].$ Relation \refx{fixpoint} is in fact nothing but
\begin{align}
-\epsil \Delta r + \epsil r + \lam  m\sol \u  +  \lam \sol( {r} \uf)  = &\:  0, \label{epsr} \\ 
 -\lam m\Delta\u + \gamma m^{\gamma-1}\nabla r + \lam \rof\ufi\cdot\nabla{\u} =&\:  \lam \vektor{G}\text{ in }  \Omega,  \\
 \u\cdot\en= 0,\qquad \en\cdot\nabla r = &\: 0, \\
 \en\cdot 2 m \D(\u)\cdot\boldsymbol{\tau}^k + f \u\cdot\boldsymbol{\tau}^k  = &\: h \text{ on } \partial\Omega.  \label{epsb}
\end{align}
We test the second equation by $\lam\u$ and the first one by $\gamma m^{\gamma-2}r$ concluding
\begin{multline}
 \lam^2 m \norm{\u}_{1,2}^2 + \epsil m^{\gamma-2}\gamma  \lnorm{\nabla r}2^2+ \epsil m^{\gamma-2}\gamma  \lnorm{  r}2^2  \\ \leq 
 \lam^2(\lnorm{\vektor{G}}{6/5}\lnorm{\u}6 +\|h\|_{W^{1/2,2}(\partial \Omega)}\|u\|_{L_2(\partial\Omega)}) + \lam m^{\gamma-2}\gamma \lnorm{\sol\uf}\infty\lnorm{r}2^2 . \label{e56}  
\end{multline}
In order to close the estimates,  the last term is estimated by  the Bogovskii operator.  This reads after using Young's inequality
\begin{equation}\label{e56a}
\gamma m^{\gamma-1}\lnorm{r}2^2  \leq C \lam^2 \bigl( m^{3-\gamma} \lnorm{\nabla\u}2^2  +   m^{3-\gamma}  \lnorm{\ufi}3^2 \lnorm{\nabla\u}2^2 +   m^{1-\gamma}\lnorm{\vektor{G}}{6/5}^2  \bigr).
\end{equation}
Incorporating this into \refx{e56}, we obtain
\begin{multline*}
 \lam^2 m \norm{\u}_{1,2}^2 + \epsil m^{\gamma-2} \lnorm{\nabla r}2^2+ \epsil m^{\gamma-2} \lnorm{  r}2^2  \\ 
 \leq C \Biggl( \frac{\lam^2}{m}(\lnorm{\vektor{G}}{6/5}^2+\|h\|_{W^{1/2,2}(\partial \Omega)}) + \frac{\lam^3}m\lnorm{\sol\uf}\infty \Bigl(   \frac{\lnorm{\nabla\u}2^2}{m^{\gamma-3}} (1+E^2)  +   \frac{\lnorm{\vektor{G}}{6/5}^2}{m^{\gamma-1}}  \Bigr)  \Biggr)
\end{multline*}
and consequently, since $E^2\lnorm{\nabla^2\uf}p\ll m^{\gamma-1}$,
\begin{equation}
 \epsil m^{\gamma-2} \lnorm{\nabla r}2^2+ \epsil m^{\gamma-2} \lnorm{  r}2^2 \leq C\bigl( \lnorm{\vektor{G}}{6/5}, \|h\|_{W^{1/2,2}(\partial \Omega)}\bigr), \label{e57}
\end{equation}
where $C$ is independent of $\epsil$ and $\lam.$
Thus, we get for given $\epsil>0$ a fixed point of $\mathcal{T}$, which satisfies estimate \refx{e57}, and then (\ref{e56a})  with $\lam=1,$ so we  pass to the limit with $\epsil\to0^+$ to get a weak solution to~\refx{rof}--\refx{bf}.

\smallskip 

  To improve the regularity of the  solution we use the method of decomposition of Novotn\'{y} and Padula \cite{NoPa94}. 
  Here we use a bootstrap method, first we perform below estimates for $p=2$, having the weak solution, and then we are allowed to do it for $p>3$.
  First, we deduce by applying $\curl$ on \refx{mof} that $\omeg$ fulfills
   \begin{align*}
   - m  \Delta \omeg =&\:\curl\bigl( - \rof\ufi\cdot\nabla\u 
    + \vektor{G}  \bigr)\text{ in }\Omega ,\\
              \sol\omeg= &\:0,\\
     \omeg\cdot\boldsymbol{\tau}^1 =&\:-  \Bigl(2\chi_2-\frac{f}m\Bigr) {\u}\cdot\boldsymbol{\tau}^2-\frac{h}m,\\ 
         \omeg\cdot\boldsymbol{\tau}^2 = &\:  \Bigl(2\chi_1-\frac{f}m\Bigr) {\u}\cdot\boldsymbol{\tau}^1+\frac{h}m\text{ on }\partial\Omega.
   \end{align*} 
  And it satisfies the following bound
  \begin{equation*}
   m \norm{\omeg}_{{1,p}} \leq  C\Bigl(  {\norm{\curl( - \rof\ufi\cdot\nabla\u 
        + \vektor{G}  )}_{(W^{1,p{\sp{\prime}}}(\Omega))^*} }
  +m\norm{ \u}_{W^{1-\frac{1}{p},p}(\partial\Omega)}+\norm{h}_{W^{1-\frac{1}{p},p}(\partial\Omega)}\Bigr).
  \end{equation*}
  As  $\PH\u$ satisfies \refx{rotdiv}, we get that
   \begin{equation*}
     m \norm{\nabla^2\PH\u}_{{p}} \leq  C\Bigl( m\norm{ \nabla\u }_{{p}}
          + \norm{\vektor{G}  }_{{p}} \Bigr. \\\Bigl. 
    +m\norm{ \u}_{W^{1-\frac{1}{p},p}(\partial\Omega)}+\norm{h}_{W^{1-\frac{1}{p},p}(\partial\Omega)}\Bigr)  .
    \end{equation*}
   Further, using the well-known vector identity $\Delta\u = \nabla\sol\u - \curl(\curl\u)$, we observe that the linearized effective viscous flux \begin{equation}
   P=\gamma m^{\gamma-2}r - 2\sol\u\label{defP}
   \end{equation}  solves
\begin{equation*}
 m \nabla P =   \vektor{G}  - \rof\ufi\cdot\nabla{ \u} -m\curl\omeg,\qquad \int_\Omega P \dx =0,\label{Pevf}
\end{equation*}
 with the estimate 
 \begin{equation*}
 m  \norm{P}_{1,p}\leq C \bigl(\lnorm{\vektor{G}}{p}  +m\lnorm{\nabla {\u}}{p}+m\lnorm{\curl\omeg}{p} \bigr).
 \end{equation*}
  Next, combining  the continuity equation \refx{rof} together with relation~\refx{defP},
we observe that  the variation of the density $r$ actually satisfies the stationary transport equation
\begin{equation}
r+\sol\Bigl(  \frac{2r \uf}{\gamma m^{\gamma-1}} \Bigr) = \frac{P}{\gamma m^{\gamma-2}}\text{ in }\Omega,\qquad \inte{r}=0. \label{STE}
\end{equation}  
  Noting that   \begin{equation}
   \frac{\norm{\uf}_{2,p}}{m^{\gamma-1}}\leq \alpha\label{fortrans}
   \end{equation}
   for some $\alpha$ sufficiently small and $ \uf\cdot\en =0 \text{ on }\partial\Omega$, we can deduce that  the unique solution $r$ of  problem \refx{STE} satisfies $$m^{\gamma-2}\norm{r}_{W^{1,p}(\Omega)}\leq C \norm{P}_{W^{1,p}(\Omega)},$$
see~\cite[Theorem 5.1]{No_CMUC}.
   
    Finally, the definition of Helmholtz decomposition yields that actually $\sol\u = \Delta\Pg\u$,  hence according to~\refx{defP} the potential part of the velocity field $\Pg\u$ satisfies the Neumann problem
  \begin{gather*} 
 -2\Delta \Pg\u = P-\gamma m^{\gamma-2} r\text{ in }\Omega,\\
  \nabla\Pg\u\cdot\en= \:0  \text{ on }\partial\Omega,
   \end{gather*} 
providing by the standard elliptic theory the estimate $$m\lnorm{\nabla\Pg\u}{2,p}\leq Cm \lnorm{P-\gamma m^{\gamma-2} r}{1,p}.$$ Therefore, summing up the estimates above, we get that solution to \refx{rof}--\refx{bf} fulfils
\begin{equation*}
m\lnorm{  \u}{W^{2,p}(\Omega)} + m^{\gamma-1}\norm{r}_{W^{1,p}(\Omega)} \leq C   \Bigl( m\norm{ \nabla\u }_{L^p(\Omega)}  +m\norm{ \u}_{W^{1-\frac{1}{p},p}(\partial\Omega)}
   + \norm{\vektor{G}  }_{L^p(\Omega)}   +\norm{h}_{W^{1-\frac{1}{p},p}(\partial\Omega)}\Bigr).
\end{equation*}
 The first two terms can be  put to the left-hand side by means of  interpolation with the energy norm, while the rest is controlled, so we see that the solution has the proposed regularity. This completes the proof of this lemma.
\end{proof}
In order to finish the proof of Proposition \ref{Prop1} we  find a fixed point of the mapping $\widetilde{\u}\mapsto\u$ defined through\footnote{Let us recall the notation $\rof=m+\rf.$}
\begin{align}
 m\sol \u  +  \sol( r \uf)  = &\:  0, \label{RKg} \\ 
 -\sol\bigl(2m\D(\u)\bigr)+ \gamma m^{\gamma-1}\nabla r  =&\:   \sol\bigl(2\rf \D(\widetilde{\u})\bigr) + \nabla R_m(\rf)+  \rof\ef - \rof\ufi\cdot\nabla {\u} \text{ in }\Omega,  \label{MEg}\\
 \u\cdot\en= &\:0, \label{BC1}\\
   \en\cdot 2 m \D(\u)\cdot\boldsymbol{\tau}^k+ f \u\cdot\boldsymbol{\tau}^k  = &\: - \en\cdot 2 \rf \D(\widetilde \u)\cdot\boldsymbol{\tau}^k \text{ on }\partial\Omega,\quad\int_{\Omega}{r\:}\dx=0. \label{BC2}
\end{align}
The mapping is according to the previous lemma well-defined from $W^{2,p}(\Omega)$ to $W^{2,p}(\Omega).$ 
We want to show that in fact it maps $M_\u(m)$ into itself and that it is a contraction. For this purpose, we test the first equation with $\gamma m^{\gamma-2}r$, the second equation with $\u$, and sum up the resulting relations. We end up with
\begin{multline*}
\inte{ 2m\abs{\D(\u)}^2 } + \sum\limits_{k=1,2}\inth{f |{ \u\cdot\boldsymbol{\tau}^k|}^2}=   \inte{ 2\rf\D(\widetilde\u):\D(\u) }\\+
 \inte{\Bigl(-  \rof\ufi \cdot\nabla \frac{\abs{\u}^2}{2}- \frac{ \gamma m^{\gamma-2}}2 r^2\sol\uf + R_m(\rf)\sol\u+\rof\ef\cdot\u\Bigr)}.
\end{multline*}
The first term on the right-hand side can be using the H\"older inequality controlled by the left-hand side, while the convective term can be estimated 
\begin{multline} \abs{\inte{\sol\bigl((m+\rf) \ufi\bigr) \frac{\abs{\u}^2}{2} }}\leq  m \inte{ \abs{\sol\ufi} {\abs{\u}^2} } +\inte{ \abs{\nabla\rf}\abs{\ufi}  {\abs{\u}^2}}  \\
\leq   C  \norm{\u}_{1,2}^2 \Bigl(  m\norm{ \sol\ufi }_p   +  \norm{\nabla \rf}_p\lnorm{\ufi}3  \Bigr) . \label{trik}
\end{multline}
The second last term on the right-hand
 side can be estimated by Young's inequality 
$$
\Big|\int_\Omega R_m(\rf) \sol \vektor{u}\, \dx\Big| \leq \frac m2 \|\sol \vektor{u}\|_2^2 + Cm^{2(\gamma-2)-1} \|\rf\|_\infty^2.
$$
Thus, using the assumptions on $\rf,\ufi$, especially $C_\ef^2\ll \min( m,m^{\gamma-1})$
\begin{equation*}
m\lnorms{\nabla\u}2^2\leq C \Bigl( m^{\gamma-2}\norm{r}_{2}^2\lnorms{\sol\uf}{\infty} +m\lnorms{\ef}{6/5} ^2 +1\Bigr). \label{solu}
\end{equation*}
In order to obtain the $L^2$-estimate of the density, we  test the momentum equation with $-\Phib$, $\Phib = \mathcal{B}\left[r\right]$, so $\lnorm{\nabla\Phib}2\leq C\lnorm{r}2 .$ This leads to
\begin{multline*}
\gamma m^{\gamma-1}\lnorm{r}2^2 \leq  m^{\gamma-2}\lnorm{\rf}\infty\lnorm{\rf}2 \lnorm{r}2 + 2m\lnorm{\nabla\u}2 \lnorm{\nabla\Phib}2 +2   \lnorm{\rf}\infty \lnorm{\nabla\widetilde\u}2 \lnorm{\nabla\Phib}2 \\
\qqquad\qqquad \qqquad+\inte{\Bigl( (m+\rf)\ufi\cdot\nabla\u\cdot\Phib - (m+\rf)\ef\cdot\Phib\Bigr)} \\
\leq  C\Bigl((m^{\gamma-2}\lnorm{\rf}2+\lnorm{\nabla\widetilde\u}2)\lnorm{\rf}\infty + m\lnorm{\nabla\u}2 +m\lnorm{\ufi}{3}\lnorm{\nabla\u}2 +m\lnorm{\ef}{6/5} \Bigr)\lnorm{r}2,
\end{multline*} 
id est
\begin{align*}
m^{\gamma-1}\lnorm{r}2^2 \leq  C\Bigl( m^{\gamma-3}  {\lnorm{\rf}\infty^2\lnorm{\rf}2^2}   + (1+E^2)\bigl( m^{3-\gamma}\lnorm{\ef}{6/5} ^2+1 +\lnorm{r}{2}^2\lnorm{\sol\uf}{\infty}  \bigr) \Bigr).
\end{align*}
Assuming $m^{\gamma-1}\gg E^2 C_\ef$, the last term can be put to the left-hand side, hence going back to \refx{solu}, we obtain	that
	\begin{equation}
	 \lnorms{\nabla\u}2^2
	 \leq C \Biggl(\biggl(m^{3(1-\gamma)}  +(1+E^2)\Bigl(\frac{\lnorm{\ef}{{6}/{5}} ^2}{m^{\gamma-1}}+\frac{1}{m^{2}}\Bigr)\biggr) \lnorm{\sol\uf}{\infty}  +\lnorm{\ef}{{6}/{5}} ^2+m^{-1} \Biggr)\leq E^2. \label{choiceE} 
		\end{equation}	
The last inequality is satisfied for properly chosen $E$ and sufficiently large $m,$ which will be chosen later, keeping in mind constraint \refx{choiceE}. Thus we have
	\begin{equation}	
	\lnorm{\nabla\u}2\leq E.\label{EIf}
	\end{equation}	

Next, we  show that $\Xi\leq C_\ef$ for $(r,\u)$ . 
Introduce
$$\Gef=- \rof\ufi\cdot\nabla\u 
 + 2  \div (\rf \D(\widetilde\u)) +\rof\ef, $$ where $\rof=\rf+m.$  First, applying $\curl$ on \refx{MEg} yields
\begin{align*}
-m\Delta\omeg  =&\: \curl\Gef\text{ in }\Omega,\\
    m \omeg\cdot\boldsymbol{\tau}^1 =&\: -   \rf \widetilde{\omeg}\cdot\boldsymbol{\tau}^1-  \Bigl(2m\chi_2-{f}  \Bigr) {\u}\cdot\boldsymbol{\tau}^2- 2\rf\chi_2\widetilde{\u}\cdot\boldsymbol{\tau}^2,\\ 
      m   \omeg\cdot\boldsymbol{\tau}^2 = &\: -\rf \widetilde{\omeg}\cdot\boldsymbol{\tau}^2+ \Bigl(2m\chi_1-{f}  \Bigr) {\u}\cdot\boldsymbol{\tau}^1+2\rf\chi_1\widetilde{\u}\cdot\boldsymbol{\tau}^1,\\
          \sol\omeg= &\:0\text{ on }\partial\Omega,
          \end{align*}
and since $\PH\u$ satisfies \refx{rotdiv}, we conclude
\begin{equation}
m\lnorm{\nabla^2\PH\u}{p} \leq  C \Bigl( \lnorm{\rf\widetilde\omeg}{1,p} + \bigl\|{\Gef}\bigr\|_{p} +  m\lnorm{\u}{1,p}+\lnorm{\rf\widetilde\u}{1,p}    \Bigr) . \label{Stokf}
\end{equation} 
Similarly, the potential part of the momentum equation \refx{MEg} reads
\begin{equation*}
\gamma m^{\gamma-1} r  + R_m(\rf) - \{ R_m(\rf)\}_{\Omega} -2 m \sol\u = \Pg\bigl(  m \Delta \PH\u+\Gef \bigr)\label{e82}
\end{equation*}
which combined with the continuity equation
 $$ m\sol \u  + \sol(r \uf )  = 0  $$
 yields 
\begin{equation*}
\gamma m^{\gamma-1} r  + R_m(\rf) - \{ R_m(\rf)\}_{\Omega} +  2 \nabla r \cdot \uf   =  -2  r\sol\uf+\Pg\bigl(   m \Delta \PH\u
+\Gef \bigr).
\end{equation*}
After differentiating,
\begin{equation}
\gamma m^{\gamma-1}\nabla r  + 2\uf \cdot\nabla\nabla r    =  - 2\nabla r\sol\uf -2 r\nabla\sol\uf - 2 \nabla \uf \nabla r - \nabla R_m(\rf)
+\nabla\Pg\bigl(    m \Delta \PH\u
+\Gef \bigr).\label{gradf}
\end{equation}
Using the same trick as in the a priori estimates part,
$$\inte{\uf\cdot\nabla\partial_kr\abs{\partial_kr}^{p-2}\partial_kr}=-\frac{1}{p}\inte{\sol{\uf}\abs{\partial_kr}^p},$$
we obtain
\begin{multline*}
  m^{\gamma-1}\lnorm{\nabla r}{p}\leq C\Bigl( \lnorm{\nabla \uf}{\infty}\lnorm{\nabla r}p  +  \lnorm{r}\infty\lnorms{\nabla \sol \uf}p +\lnorm{\nabla r}p\lnorms{ \sol \uf}\infty\Bigr.\\\Bigl.
  + \lnorm{  \nabla R_m(\rf)}p +
  \lnorms{\Gef}p + m  \lnorms{\nabla^2\PH\u}p\Bigr),
\end{multline*}
hence since $m^{\gamma-1}\gg\Xi$, $\int_\Omega r\, \dx =0$, 
\begin{equation*}
\lnorm{\nabla r}{p}\leq  \frac C{m^{\gamma-1}}\Bigl(  \lnorm{  \nabla R_m(\rf)}p+  \lnorms{\Gef}p +  m  \lnorms{\nabla^2\PH\u}p\Bigr).
\end{equation*}
Moreover, using \refx{e82}, we bound the potential part of the velocity. As
\begin{equation*}
2m \nabla\sol\u =  \gamma m^{\gamma-1} \nabla r +\nabla R_m(\rf) - \nabla \Pg\bigl(   m \Delta \PH\u
 + \Gef\bigr),
\end{equation*}
due to the fact that $\|\nabla^2 P_\nabla \u\|_{1,p} \leq C \|\sol \u\|_{1,p} \leq \widetilde{C} \|\nabla \sol \u\|_p$, we obtain
\begin{equation*}
\Xi \leq  \frac C{m}
  \Bigl( \lnorm{\nabla R_m(\rf)}p+ \lnorm{\rf\widetilde\omeg}{1,p} + \bigl\|{\Gef}\bigr\|_{p} +  m\lnorm{\u}{1,p}+\lnorm{\rf\widetilde\u}{1,p}    \Bigr). 
\end{equation*}
According to  $C_\ef^2\ll m$, the only problematic term in $\Gef$ is again the convective term. At this point we use  that $\ufi$ satisfies the energy inequality, so
\begin{equation*}
\lnorms{\rof\ufi\cdot\nabla\u}p\leq \: \lnorm{\rof}\infty\lnorm\ufi6\lnorms{\nabla\u}{\frac{6p}{6-p}}  \leq  \: \bigl( m+\lnorm\rf\infty \bigr)\lnorm\ufi6\lnorms{\nabla\u}2^
{\frac{6-p}{3p}}\lnorms{\nabla\u}\infty^
{\frac{4p-6}{3p}}
\leq C m E^{\frac{2p+6}{3p}}\Xi^{\frac{4p-6}{3p}}.
\end{equation*}
Thus,
\begin{equation*}
\Xi\leq  C \Bigl(1+\lnorms{\ef}p + E\Xi^{\frac{4p-6}{3p}}  \lnorms{\ef}{6/5}^{\frac{2p+6}{3p}}\Bigr).\label{h50}
\end{equation*}
As $\frac{4p-6}{3p}<1$ for $p<6$,  we conclude finally
\begin{equation}
 m^{\gamma-2}\bigl(\norm{r}_{1,p}+ \norm{r}_{\infty} \bigr)+ \norm{\u}_{2,p} + \norm{\nabla\u}_\infty + \norm{ \u}_\infty 
 \leq  C \Bigl(1+\norm{\ef}_p + \norm{\ef}_{6/5}^{\frac{2p+6}{6-p}} E^{\frac{3p}{6-p}}\Bigr),\label{h51}
\end{equation}
where $C$ is an absolute constant independent of the solution, provided $\Xi\ll m$. It is sufficient to set $m$ to be appropriately greater than the right-hand side of \refx{h51} --- let us denote it by $C_\ef$. Having
in mind that restriction \refx{choiceE} has to be fulfilled, we take
\begin{equation}
\frac{\min(m,m^{\frac{\gamma-1}4})}{{\alpha}^{-1}+15}> \max\bigl(C_\ef, C_\ef^2 , C_\ef E^2, C_\ef^2 E^2, C_1, C_2\bigr)\cdot \max(C_P , C_K, C_E, C_B) ,\label{smallness}
\end{equation}  
where $C_1$ is from \refx{contr2}, $C_2$ from \refx{contraction}, $\alpha$ represents the smallness constant in \refx{fortrans}, and $C_P$,  $C_K$ and  $C_E$ denotes the constant from the Poincar\'{e}, Korn embedding ($W^{1,p}\hookrightarrow L^{\infty}$) inequality, respectively. The symbol $C_B$ stands for the constant induced by the Bogovskii operator.
 Looking back to the continuity equation, we  conclude from \refx{h51} that $ \lnorm{\sol\u}p \leq  2\frac{C_\ef^2}{m^{\gamma-1}}. $

 \smallskip 
 
Now let us prove that mapping $\widetilde{\u}\mapsto\u$ is in fact a contraction.  Indeed, 
differences of two solutions $\up=\u_1-\u_2$, $\rp=r_1-r_2$ corresponding to $\widetilde{\up}=\widetilde\u_1-\widetilde\u_2$ satisfy
 \begin{gather}
  m\sol \up  +  \sol( \rp \uf_1  ) + \sol (r_2 \widetilde\up)  = 0,  \\ 
 \rof\,\ufi\cdot\nabla\up -2m\sol\bigl(\D(\up)\bigr)-\sol\bigl(2\rf \D(\widetilde{\up})\bigr)+ \gamma m^{\gamma-1}\nabla \rp  =\vektor{0}  \text{ in }\Omega, \label{ME}  \\
     \up\cdot\en = \: 0,\quad
     \en\cdot 2m\D(\up)\cdot\boldsymbol{\tau}^k + f \up\cdot\boldsymbol{\tau}^k=-\en\cdot 2\rf\D(\widetilde\up)\cdot\boldsymbol{\tau}^k  \text{ on }\partial\Omega.
 \end{gather}
 Basic energy estimate reads
\begin{multline*}
\inte{ 2m\abs{\D(\up)}^2 } + \sum\limits_{k=1}^{2} \inth{f \bigl|{\up\cdot \boldsymbol{\tau}^k }\bigr|^2 }=  \inte{  2\rf\D(\widetilde\up):\D(\up) } \\
=  \inte{\Bigl( -  \rof\ufi \cdot\nabla \frac{\abs{\up}^2}{2}- \frac{ \gamma m^{\gamma-2}}2 \rp^2\sol\uf_1+  \sol (r_2 \widetilde\up) \gamma m^{\gamma-2} \rp \Bigr)}.
\end{multline*}
Further, using again \refx{trik}, we obtain
\begin{equation}
m\lnorm{\nabla\up}2^2\leq C \Bigl( m^{\gamma-2} \lnorm{\rp}{2}^2\lnorm{\sol\uf_1}{\infty} + \frac{\lnorm{\rf}{\infty}^2}{m} \lnorms{\nabla\widetilde\up}2^2 + m^{\gamma-2} \lnorms{  \sol (r_2 \widetilde\up)  }2\lnorm{R}2 \Bigr)  . \label{e94} 
\end{equation}
Estimating the density using the Bogovskii operator leads to
\begin{align*}
\gamma m^{\gamma-1}\lnorm{\rp}2^2 \leq &  \: 2m\lnorm{\nabla\up}2\lnorm{\nabla\Phib}2  +    \lnorm{\rf}\infty\lnorms{\nabla\widetilde\up}2\lnorm{\nabla\Phib}2 
+\inte{\Bigl(  (m+\rf)\ufi\cdot\nabla\up\cdot\Phib \Bigr)} \\
\leq &\: C\Bigl( m\lnorms{\nabla\up}2 + \lnorm{\rf}\infty\lnorms{\nabla\widetilde\up}2 +m\lnorm{\ufi}{3}\lnorms{\nabla\up}2  \Bigr)\lnorm{\rp}2,
\end{align*} 
hence by Young's inequality and \refx{e94}
\begin{multline*}
m^{\gamma-1}\lnorm{\rp}2^2 \leq  C \Bigl( m^{3-\gamma}\lnorm{\nabla\up}2^2 + \frac{\lnorm{\rf}\infty^2\lnorms{\nabla\widetilde\up}2^2}{m^{\gamma-1}} +m^{3-\gamma} E^2\lnorm{\nabla\up}2^2  \Bigr)\\
\leq C (1+E^2)\Bigl( \lnorm{\rp}{2}^2\lnorm{\sol\uf_1}{\infty}+\lnorms{  \sol (r_2 \widetilde\up) }2 \lnorm{\rp}2 \Bigr) + C \frac{\lnorm{\rf}{\infty}^2}{m^{\gamma-1}} \lnorms{\nabla\widetilde\up}2^2  .
\end{multline*} 
As $\lnorm{\sol\uf_1}\infty\ll m^{\gamma-1} $, the first term can be put to the left-hand side, so we get again by Young's inequality
\begin{align*}
m^{\gamma-1}\lnorm{\rp}2^2 \leq C \biggl( (1+E^4)\frac{\lnorms{  \sol (r_2 \widetilde\up) }2^2}{m^{\gamma-1}}  + \frac{\lnorm{\rf}{\infty}^2}{m^{\gamma-1}} \lnorms{\nabla\widetilde\up}2^2 \biggr) .
\end{align*} 
Since $\lnorms{  \sol (r_2 \widetilde\up) }2\leq \lnorm{ r_2 }\infty \lnorms{\nabla\widetilde{\up}} 2  +   \lnorm{ \nabla r_2 }3 \lnorms{ \widetilde{\up}} 6 , $  we conclude
\begin{align*}
m^{\gamma-1}\lnorm{\rp}2 \leq C  ( 1+ E^2)  C_\ef \lnorms{\nabla\widetilde\up}2  .
\end{align*} 
Going back to \refx{e94} we obtain almost final form of the desired estimate
\begin{equation*}
m\lnorm{\nabla\up}2^2\leq C \bigl(1+E^2\bigr) \Bigl( \frac{C_\ef^2}{m^\gamma} \lnorms{\nabla\widetilde\up}2^2\lnorm{\sol\uf_1}{\infty} + \frac{C_\ef^2}{m} \lnorms{\nabla\widetilde\up}2^2 \Bigr) . \label{contr} 
\end{equation*}
Therefore, for $m$ sufficiently large, we  write for some $C_1$ which is independent of $m$,
\begin{equation}
\lnorms{\nabla\up}2\leq \frac{C_1}{m} \lnorms{\nabla \widetilde{\up}}2.\label{contr2}
\end{equation}
Taking $m>C_1$, we obtain that the mapping is contraction in the $W^{1,2}$-metric. Thus, using  the boundedness in $M_\u(m)\subset W^{2,p}(\Omega)$ as well, Proposition \ref{Prop1} is proved.\footnote{More precisely, we use the following simple consequence of standard Banach contraction argument and  weak compactness of reflexive spaces.
\begin{thm*} 
 Let $X$, $Y$ be Banach spaces such that $X$ is reflexive and continuously embedded into $Y$ ($X \hookrightarrow Y$), let $K\subset X$ be
 a non-empty,  convex, bounded subset of $X$. Suppose further that $\mathcal{T}\colon K \to K$  is a contraction mapping  in $Y$-metric, id est
  $$  \norm{\mathcal{T}(u)-\mathcal{T}(v)}_Y\leq \kappa \norm{u-v}_Y ,\quad\forall u,v \in K,$$
  for some $0\leq \kappa<1.$ Then $\mathcal{T}$ possesses a unique fixed point in $K$.
 \end{thm*}.}
\end{proof}

\subsubsection*{Elimination of the density linearization}
\begin{prop}\label{Prop2}
Suppose $\ufi\in M_\u(m)$ for $m$  sufficiently large, then there exists a unique solution $(r,\u)$ to problem \refx{RKgiv}--\refx{MEgiv} with boundary conditions \eqref{imper}--\eqref{slip} in the class $M_r(m)\times M_\u(m)$.
\end{prop}

\begin{proof}
We apply  the Banach contraction principle on the mapping $\mathcal{S}_\ufi:M_r(m) \to M_r(m),$
defined as a solution operator to the following problem $\mathcal{S}(r_n )=r_{n+1}$
\begin{align}
 m\sol \u_{n+1}  + \sol( r_{n+1} \u_{n+1} )   = \: & 0, \\ 
(m+r_n)\,\ufi\cdot\nabla\u_{n+1} -\sol\bigl(2(m+r_n)\D(\u_{n+1})\bigr) + \gamma m^{\gamma-1}\nabla r_{n+1}  + \nabla&  R_m(r_n)= (m+r_n)\ef \text{ in }\Omega, \label{MEn}\\
 \u_{n+1}\cdot\en= \: & 0,\nonumber \\
   \en\cdot 2 (m + r_{n}) \D(\u_{n+1})\cdot\boldsymbol{\tau}^k+ f \u_{n+1}\cdot\boldsymbol{\tau}^k  = \: & 0 \text{ on }\partial\Omega,\quad\int_{\Omega}{r_{n+1}\:}\dx=0.
\end{align}
The solvability of system \refx{MEn} in $M_r(m)\times M_\u(m)$ was proven in Proposition~\ref{Prop1}. Thus, $\mathcal{S}$ indeed maps $M_r(m)$ into itself.
We  show that $\mathcal{S}$ is contraction. Let us denote $$\u=\u_{n+1}-\u_n,\:r= r_{n+1}-r_n,\:r_-= r_{n}-r_{n-1},$$  
then the difference $(\u,r)$ satisfies
\begin{align}
m\sol\u +\sol(r\u_{n+1}) + \sol( r_n \u) = \: & 0   \nonumber\\
(m+r_n)\ufi \cdot\nabla \u + r_- \ufi\cdot \nabla \u_n - \sol[(2(m+r_n)\D(\u)) + (2r_-\D(\u_n))]  & \nonumber\\  +  \gamma m^{\gamma-1}\nabla r + \nabla \big(R_m(r_n)-R_m(r_{n-1})\big) = \: & r_-\ef \text{ in }\Omega,\nonumber\\
 \u\cdot\en= \: & 0,\nonumber \\
   \en\cdot[ 2 (m + r_{n}) \D(\u)+  2  r_{-} \D(\u_n)]\cdot\boldsymbol{\tau}^k+ f \u\cdot\boldsymbol{\tau}^k  = \: & 0 \text{ on }\partial\Omega. \label{systemS}
\end{align}
First, let us test the momentum equation of \refx{systemS} by the difference $\u$ and the continuity equation by $\gamma m ^{\gamma-2}r$, this turns after usage of H\"{o}lder's and Young's inequalities into
\begin{multline*}
m\lnorms{\nabla\u}2^2\leq C\biggl( \frac{\lnorm{r_-}2^2 }{m}  \lnorm{\nabla\u_n}\infty^2(1+E^2) +m^{\gamma-2} \lnorm{r_n+r_{n+1}}\infty\lnorm{r_-}2\lnorm{\sol\u}2 \biggr.\\\biggl. +  m^{\gamma-2} \Bigl(\lnorm{\sol\u_{n+1}}\infty \lnorm{r}2^2 +\lnorm{\sol(r_n\u)}2\lnorm{r}2\Bigr)+ \frac{\lnorm{r_-}2^2\lnorm{\ef}3^2}{m}\biggr)   .\end{multline*}
The second term on the right-hand side can be put directly to the left-hand side, and similarly we proceed with the other term containing $\u$, this leads to
\begin{multline}
m\lnorms{\nabla\u}2^2\leq C\biggl( \frac{\lnorm{r_-}2^2 }{m}  \lnorm{\nabla\u_n}\infty^2(1+E^2) +  m^{\gamma-2}  \lnorm{\sol\u_{n+1}}\infty \lnorm{r}2^2 \biggr.\\\biggl.  + m^{2\gamma-5} (\|r_n\|_\infty + \|r_{n+1}\|_\infty)^2\|\|r_{-}\|_2^2  +\frac{\lnorm{r_n}\infty +\lnorm{\nabla r_n}3}{m^{5-2\gamma}} \lnorm{r}2^2+ \frac{\lnorm{r_-}2^2\lnorm{\ef}3^2}{m}\biggr). \label{e101}  \end{multline}
Further, using the Bogovskii type of estimates we obtain
\begin{multline*}
 m^{\gamma-1}\lnorm{r}2^2 \leq C \Bigl( m\lnorm{\ufi}3 \lnorm{\nabla\u}2 + \lnorm{r_-}2\lnorm{\ufi}3\lnorm{\nabla \u_n}\infty   + m\lnorm{\nabla\u}2 \Bigr. \\ \Bigl.  + \lnorm{r_-}2\lnorm{\nabla\u_n}\infty +m^{\gamma-2}\lnorm{r_n+r_{n+1}}\infty \lnorm{r_-}2 + \lnorm{r_-}2\lnorm{\ef}3 \Bigr)\lnorm{r}2  
\end{multline*}
and by means of Young's inequality
\begin{multline*}
 m^{\gamma-1}\lnorm{r}2^2 \leq C \biggl( \frac{\lnorm{\ufi}3^2 \lnorm{\nabla\u}2^2}{m^{\gamma-3}} + \frac{\lnorm{r_-}2^2\lnorm{\ufi}3^2\lnorm{\nabla \u_n}\infty^2 }{m^{\gamma-1}} + \frac{\lnorm{\nabla\u}2^2}{m^{\gamma-3}} \biggr. \\ \biggl.  + \frac{\lnorm{r_-}2^2\lnorm{\nabla\u_n}\infty^2}{m^{\gamma-1}} +\frac{\lnorm{r_n+r_{n+1}}\infty^2 \lnorm{r_-}2^2}{m^{3-\gamma}} + \frac{\lnorm{r_-}2^2\lnorm{\ef}3^2}{m^{\gamma-1}} \biggr)  .
\end{multline*}
Next, using \refx{e101}
\begin{equation*}
 m^{\gamma-1}\lnorm{r}2^2 \leq  C(\ef) {\lnorm{r_-}2^2} \Bigl(\frac{1}{m^{\gamma-1}} + \frac{1}{{m}} \Bigr)  +(1+E^2)\Bigl( C_\ef+ 2\frac{C_\ef}{m}  \Bigr) \lnorm{r}2^2  .
\end{equation*}
 The last term can be put to the left-hand side, while the rest is controlled, hence we obtain
 \begin{equation}
 \lnorm{r}2^2 \leq \frac{C\bigl(\lnorm{\ef}{3}\bigr)}{m^{\gamma-1}} \lnorm{r_-}2^2 = \frac{C_2}{m}\lnorm{r_-}2^2.  \label{contraction}
 \end{equation} 
The mapping is contraction for  $m>C_2$ and bounded in $M_r(m)\subset W^{1,p}(\Omega)$, so we obtain a unique solution in $M_\u(m)\times M_r(m)$ using the same type of contraction result as above.
\end{proof}

\subsubsection*{Elimination of the velocity linearization}
We now consider \refx{RKgiv}--\refx{MEgiv} with boundary conditions \eqref{imper}--\eqref{slip}.
The last step consists in proving that the map $\mathcal{T}(\ufi)=\u$  possesses a fixed point. This will be proved by applying the Schauder fixed point theorem. The previous propositions yield that $\mathcal{T}$ maps $\Mdiv$ into $M_\u(m)$. Since $M_\u(m)\subset \Mdiv $, $\Mdiv$ is a convex and closed subset of $W^{1,\infty}(\Omega)$ and $M_\u(m)$ is a compact subset of $W^{1,\infty}(\Omega)$, it remains to show that $\mathcal{T}$ is continuous on $\Mdiv$. Let us take $\ufi_1,\:\ufi_2$ and the corresponding solutions $(r_1,\u_1)$ and $(r_2,\u_2)$. We would like to estimate $r=r_1-r_2$ and $\u=\u_1-\u_2$ by means of $\ufi=\ufi_1-\ufi_2.$
We have for $k=1,2$
\begin{gather*}
 m\sol \u_k  + \sol( r_k \u_k) = 0,  \\ 
(m+r_k)\ufi_k\cdot\nabla\u_k -\sol\bigl(2(m+r_k)\D(\u_k)\bigr)+ \gamma m^{\gamma-1}\nabla r_k  + \nabla R_m(r_k)=(m+r_k)\ef  . 
\end{gather*}
Taking the difference yields
\begin{align*}
 m\sol \u  + \sol( r \u_1)+  \sol( r_2 \u) =&\:0,  \\ 
(m+r_1)\,\ufi_1\cdot\nabla\u+ (m+r_1)\,\ufi\cdot\nabla\u_2+ r\,\ufi_2\cdot\nabla\u_2 -\sol\bigl[2(m+r_1)\D(\u) +(2r\D(\u_2)) \bigr]\qqquad&\nonumber \\+ \gamma m^{\gamma-1}\nabla r  + \nabla \bigl(R_m(r_1)-R_m(r_2) \bigr)= &\: r\ef  . 
\end{align*}
Further
\begin{align*}
\u \cdot \vektor{n} &=0, \\
\vektor{n} \cdot [ 2(m+r_1)\D(\u) + 2r \D(\u_2) ]\cdot \boldsymbol{\tau}_k + f \u\cdot \boldsymbol{\tau}_k &=0
\end{align*}
on $\partial \Omega$.
The standard energy estimate reads
\begin{multline*}
(m-\lnorm{r_1}\infty)\lnorm{\nabla\u}2^2\leq C \bigl((m\lnorm{\ufi}3\lnorm{\nabla\u_2}2+ \lnorm{r}2 \lnorm{\ufi_2}\infty\lnorm{\nabla\u_2}3)\lnorm{\u}6\bigr.\\
\bigl.\qqquad\quad + \lnorm{r}2\lnorm{\ef}3\lnorm{\u}6 +
\lnorm{r}2\lnorm{\nabla\u_2}\infty\lnorm{\nabla\u}2+m^{\gamma-2}
\lnorm{\sol\u_1}\infty \lnorm{r}2^2\bigr.\\
\bigl. +m^{\gamma-2} \lnorm{\sol( r_2 \u) }2 \lnorm{r}2 \bigr),
\end{multline*}
where we have used the fact that the first term coming from the convective term can be rewritten
$$\inte{(m+r_1)\,\ufi_1\cdot\nabla\frac{\abs{\u}^2}{2}}  = -\inte{ \Bigl( (m+r_1) \sol\ufi_1  \frac{\abs{\u}^2}{2} + \ufi_1 \cdot\nabla r_1  \frac{\abs{\u}^2}{2}\Bigr) }$$
and pushed to the left-hand side, as well as the term from the nonlinear part of the pressure. Thus, after systematic usage of Young's inequality we end up with
\begin{equation}
\begin{split}
m\lnorms{\nabla\u}2^2\leq &\: C \Bigl(m\lnorm{\ufi}3^2\lnorms{\nabla\u_2}2^2+ \frac{\lnorm{r}2^2 \lnorm{\ufi_2}\infty^2\lnorms{\nabla\u_2}3^2}{m} + \frac{\lnorm{r}2^2\lnorms{\ef}3^2}{m}\bigr.\\
\bigl.&  +
\frac{\lnorm{r}2^2\lnorms{\nabla\u_2}\infty^2}{m}+
m^{\gamma-2}\lnorms{\sol\u_1}\infty \lnorm{r}2^2 +m^{\gamma-2}\lnorms{\nabla r_2}p\lnorms{\nabla \ufi }2 \lnorm{r}2 \Bigr)\\
\leq &\: C m \lnorms{\nabla\ufi}2^2 + C m^{\gamma-2} \lnorm{r}2^2 \bigl(C_\ef^2 E^2 +1 \bigr)    .\label{EIdif}
\end{split}
\end{equation}
Next, we use as usually the test function $\Phib=\mathcal{B}[r]$ in the momentum equation to get
\begin{multline*}
m^{\gamma-1}\lnorm{r}2^2\leq C\bigl( m\lnorm{\ufi_1}3\lnorm{\nabla\u}2 \lnorm{\Phib}6+ m\lnorm{\ufi}3 \lnorm{\nabla\u_2}2\lnorm{\Phib}6+ \lnorm{r}2\lnorm{\ef}3\lnorm{\Phib}6 \\
+ \lnorm{r}2\lnorm{\ufi_2}\infty \lnorm{\nabla\u_2}3\lnorm{\Phib}6
+2m \lnorm{\nabla\u}2 \lnorm{\nabla\Phib}2 + \lnorm{r}2\lnorm{\nabla\u_2}\infty\lnorm{\nabla\Phib}2  \bigr)
\end{multline*}
and using $C_{\ef}^2,\lnorm{\ef}3\ll m^{\gamma-1}$
\begin{equation}
m^{\gamma-1}\lnorm{r}2^2\leq C m ^{3-\gamma}\bigl( \lnorm{\ufi_1}3^2\lnorm{\nabla\u}2^2+ \lnorm{\ufi}3^2 \lnorm{\nabla\u_2}2^2 \\
+ \lnorms{\nabla\u}2^2  \bigr).\label{Bodif}
\end{equation}
Combining \refx{EIdif} and \refx{Bodif} yields, using  once more that $C_\ef^2E^2\ll m$
\begin{align*}
m\lnorm{ \u}{1,2}^2
\leq  C (m,\ef) \lnorms{ \ufi}{1,2}^2 .
\end{align*}
Moreover, we can use the higher order estimate following from the previous construction
 \begin{align*}
m\lnorm{ \u}{W^{2,p}(\Omega)}^2
\leq  C (m,\ef)\bigl( \lnorm{ \ufi_1}{W^{1,\infty}(\Omega)}^2 + \lnorm{ \ufi_2}{W^{1,\infty}(\Omega)}^2 +1\bigr).
\end{align*} 
in order to interpolate
\begin{align*} \lnorm{\u}{{1,\infty}}\leq &\: C  \lnorm{\u}{{1,2}}^{\alpha}  \lnorm{\u}{{2,p}}^{1-\alpha} \\
\leq &\: C(m,\ef) \lnorm{\ufi}{{1,2}}^{\alpha} \bigl( \lnorm{\ufi_1}{{1,\infty}}^{1-\alpha}+\lnorm{\ufi_2}{{1,\infty}}^{1-\alpha}   +1\bigr) 
\end{align*}
for some $\alpha\in(0,1)$,
yielding the desired continuity in $W^{1,\infty}(\Omega).$
Thus, we apply the Schauder fixed point theorem, which completes the proof of our main result. Theorem \ref{main} is done.

\subsection*{Acknowledgement} 
The work on this paper was partially conducted during the first author's internship at the Warsaw Center of Mathematics and Computer Science. The first and the third author were supported by Czech Science Foundation (grant no. 16-03230S).
 The second author (PBM) has been partly supported by National
Science Centre grant 2014/14/M/ST1/00108 (Harmonia).

  

\end{document}